\documentclass[reqno, 11pt, a4paper]{amsart}
\usepackage{amsfonts}
\usepackage{graphicx}
\usepackage{amscd,color}
\usepackage{amsmath}
\usepackage{amssymb}
\usepackage{mathtools} 

\usepackage{eucal} 
\usepackage{subfig} 
\usepackage[pdftex]{hyperref}  

\usepackage[T1]{fontenc} 
\usepackage[utf8]{inputenc} 
\usepackage[english]{babel}

\usepackage{cancel}  

\newtheorem{theoremletter}{Theorem}

\newcommand{\Rset}{\mathbb{R}}

\newcommand{\Div}{\textrm{div}}

\def\<{\left\langle}
\def\>{\right\rangle}

\newtheorem{teorema}{Theorem}[section]

\newtheorem{corolario}[teorema]{Corollary}
\newtheorem{lema}[teorema]{Lemma}
\newtheorem{definicao}[teorema]{Definition}
\newtheorem{observacao}[teorema]{Remark}

\newenvironment{demonstracao}[1][Proof]{\noindent\textbf{#1:} }{\ \hspace{\stretch{1}}$\square$\\}

\title[Serrin's type Problems in convex cones in Riemannian manifolds]{Serrin's type Problems in convex cones in Riemannian manifolds}
\author[M. Araújo, A. Freitas, M. Santos and J. S. Sindeaux]{Murilo Araújo$^{1,2}$, Allan Freitas$^{\ast, 1}$, Márcio Santos$^1$ and Joyce S. Sindeaux$^{3}$} 
\address{$^1$Departamento de Matem\'{a}tica\\
	Universidade Federal da Para\'{\i}ba\\
 	58.051-900 Jo\~{a}o Pessoa, Para\'{\i}ba, Brazil}
 \address{$^2$Universidade Federal do Agreste de Pernambuco\\
 	55292-278 Garanhuns, Pernambuco, Brazil} 
 \address{$^3$Universidade Regional do Cariri\\
 	63150-000 Campos Sales, Ceará, Brazil} 
 \email{allan@mat.ufpb.br}
 \email{murilo.chavedar@academico.ufpb.br}
 \email{marcio.santos@academico.ufpb.br}
 \email{joyce.sindeaux@urca.br}


\thanks{$^\ast$Corresponding author.}

\begin{document}

\begin{abstract}
In this work, we discuss several results concerning Serrin's problem in convex cones in Riemannian manifolds. First, we present a rigidity result for an overdetermined problem in a class of warped products with Ricci curvature bounded below. As a consequence, we obtain a rigidity result for Einstein warped products. Next, we derive a soap bubble result and a Heintze-Karcher inequality that characterize the intersection of geodesic balls with cones in these spaces. Finally, we analyze the analogous ovedetermined problem for the drift Laplacian, where the ambient space is a cone in the Euclidean space.
\end{abstract}

\maketitle


\section{Introduction} 
The classical Serrin overdetermined problem \cite{Serrin1971} asserts that a solution exists to the boundary value problem 
\begin{eqnarray} \label{firstSP}
    \left\lbrace \begin{array}{ll}
	\Delta u   & = -1 \quad \textrm{in } \Omega, \\[5pt]
	u & = 0 \quad \textrm{on } \partial\Omega, \\[5pt]
    u_{\nu} & = -c \quad \textrm{on } \partial\Omega,
    \end{array} \right. 
\end{eqnarray}
where $\Omega \subset \mathbb{R}^n$ is a bounded domain, $\nu$ denotes the outward unit normal vector on $\partial \Omega$, and $c > 0$ is a constant,  
if and only if $\Omega$ is a ball and $u$ is a radial function. Serrin provided the complete solution to this problem in \cite{Serrin1971}, where he also presented remarkable motivation arising from fluid dynamics (see also \cite{soko} and \cite[Section 2]{dpv}). While the classical method to demonstrate the rigidity of \eqref{firstSP} relies on the use of the moving planes technique, an alternative approach was introduced in the same published edition of \cite{Serrin1971}. Specifically, Weinberger \cite{Weinberger1971} proposed the $P$-function method, based on an integral technique, to provide an alternative proof of the symmetry result associated with \eqref{firstSP}. This method will be discussed later in more detail.

Related to \eqref{firstSP}, the classical radial solution on the Euclidean ball centered at the origin $o$ with radius $R$, denoted by $B_{R}(o)$, is given by 
\begin{equation}\label{u_sol}
    u(x) = \frac{R^2 - |x|^2}{2n},
\end{equation}
where $|x|$ denotes the Euclidean norm of $x \in \mathbb{R}^n$ and $R=cn$. Considering the open cone in $\mathbb{R}^n$
$$\Sigma\doteq\left\{tx;\\ x\in\omega,\\ t\in (0,\infty)\right\},$$
where $\omega$ is an open connected domain on the unit sphere $\mathbb{S}^{n-1}$, it is worth to note (as observed in \cite{CiraoloRoncoroni2020}) that \eqref{u_sol} also is a solution for the following partially overdetermined problem
\begin{eqnarray*} 
    \left\lbrace \begin{array}{ll}
	\Delta u   & = -1 \quad \textrm{in } B_{R}(o)\cap \Sigma, \\[5pt]
	u & = 0 \quad \textrm{on } \partial B_{R}(o)\cap \Sigma, \\[5pt]
    u_{\nu} & = -c \quad \textrm{on } \partial B_{R}(o)\cap \Sigma,\\[5pt]
    u_{\nu} & = 0 \quad \textrm{on } B_{R}(o)\cap \partial\Sigma.
    \end{array} \right. 
\end{eqnarray*}
As in Serrin's classical result, it is natural to investigate symmetry characterizations, both for the domains and for the solutions, in the sense of {\it sector-like domains}, which, in this specific characterization, can be viewed as intersections between a ball and a cone. These domains will be defined below in a quite general way (see Definition \ref{def_sec}).

In this context, Pacella and Tralli (\cite{Pacella2020}), as well as Ciraolo and Roncoroni (\cite{CiraoloRoncoroni2020}), have extended the study of rigidity characterizations for partially overdetermined problems by employing integral methods. While \cite{Pacella2020} focus on the partially overdetermined problem in a sector-like domain in the Euclidean space, \cite{CiraoloRoncoroni2020} consider more general operators than the Laplacian in the Euclidean space, including possibly degenerate operators, and also addressing analogous problems in space forms, the hyperbolic space and the (hemi)sphere. See also \cite{leeseo} for related overdetermined problems concerning cones in spheres.

Along this last direction, a noteworthy aspect in addressing these problems within the setting of general Riemannian manifolds is that the associated equation  
\[
\Delta u + nku = -1,
\]
where \((M, g)\) is a Riemannian manifold with Ricci curvature bounded below by \(\operatorname{Ric} \geq (n-1)k g\), leads to certain rigidity results. This is primarily due to the application of the \(P\)-function method, an approach that utilizes a carefully defined \(P\)-function that is subharmonic, allowing one to exploit the classical maximum principle as in Weinberger's technique. For further details, we refer the reader to the original work by \cite{CiraoloVezzoni2019}, as well as subsequent developments in \cite{FR_2022}, \cite{CiraoloRoncoroni2020}, and \cite{FAM_2023}, among others. This framework effectively enables the method to yield significant results, even in the case of spherical domains. 

In this work we extend the analysis of such symmetry characterization of sector-like domains for Riemannian manifolds in some directions. First, motived by \cite{CiraoloRoncoroni2020}, we consider a warped product, this is, $M^{n}= I \times N^{n-1}$ equipped with the metric
$$g=dt^{2}+\rho^{2}(t)g_{N},$$
where $I=[0,R)$,  $R\in (0,\infty]$, $(N,g_{N})$ is a $(n-1)$-dimensional Riemannian manifold and $\rho:I\longrightarrow\mathbb{R}$ is a positive smooth function. As in the standard notation, we denote this structure by $M\doteq I\times_{\rho} N$. We point out that the space forms are classical examples of such structures: by choosing $(N^{n-1},g_{N})=(\mathbb{S}^{n-1},g_{st})$, the standard $(n-1)$-sphere, we obtain
\begin{itemize}
    \item The Euclidean space $\mathbb{R}^{n}$: by doing $\rho(t)=t$  and $I = [0, \infty)$;
    \item The Hyperbolic space $\mathbb{H}^{n}$: by doing $\rho(t)=\sinh (t)$  and $I = [0, \infty)$;
    \item The (upper) hemisphere $\mathbb{S}^{n}_{+}$: by doing $\rho(t)= \sin(t)$  and $I = [0, \pi/2)$.
\end{itemize}

If we denote $O$ the pole of the model, we define an \emph{open cone} $\Sigma$ with vertex at $\left\{O\right\}$ as the set
$$\Sigma \doteq \left\{tx;\\ x\in\omega,\\ t\in I\right\},$$
for some open domain $\omega\subset N$. Moreover, we say that a cone is \emph{convex} if its second fundamental form is nonnegative at every point \( x \in \partial \Sigma \). Here, if \( \nu \) denotes the outward unit normal vector field, the second fundamental form is given by $A(X, Y) = g(\nabla_{X} \nu, Y),$ where \( X \) and \( Y \) are tangent vector fields to \( \partial \Sigma \). In the same way as \cite{CiraoloRoncoroni2020} and \cite{Pacella2020} we have the concept of \emph{sector-like domain} that we define below.

\begin{definicao}\label{def_sec}
Let $\Sigma$ be an open cone as described above, such that $\partial\Sigma \setminus \{O\}$ is smooth. A bounded domain $\Omega \subset \Sigma$ is called a {\it sector-like domain} if  
$$
\Gamma = \partial\Omega \cap \Sigma \quad \mbox{and} \quad \Gamma_{1} = \partial\Omega \setminus \overline{\Gamma},
$$  
are such that the $(n-1)$-Hausdorff measures satisfy $\mathcal{H}_{n-1}(\Gamma) > 0$, $\mathcal{H}_{n-1}(\Gamma_{1}) > 0$, and $\Gamma$ is a smooth $(n-1)$-dimensional manifold, while $\partial\Gamma = \partial\Gamma_{1} \subset \partial\Omega \setminus \{O\}$ is a smooth $(n-2)$-dimensional manifold.
\end{definicao}
    
\begin{figure}[h] 
  \centering
  \includegraphics[scale=0.4]{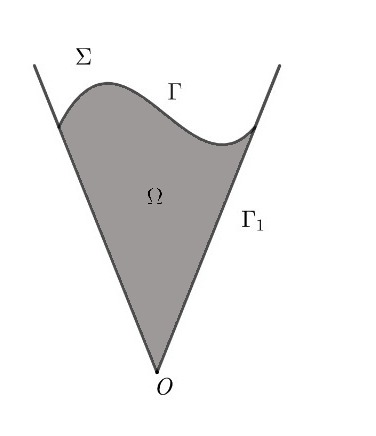}
  \caption{A sector-like domain $\Omega$ inside $\Sigma$ }
  \label{sector-like}   
\end{figure}

For such domains, we study the following overdetermined problem

\begin{eqnarray} \label{SP_cone1}
    \left\lbrace \begin{array}{ccccc}
	\Delta u + nku  & = & -1 & \textrm{ in } & \Omega \\
	u & > & 0 & \textrm{ in } & \Omega \\
    u & = & 0 & \textrm{ on } & \Gamma \\
        u_{\nu} & = & -c & \textrm{ on } & \Gamma \\
        u_{\nu} & = & 0 & \textrm{ on } & \Gamma_1 \setminus \{O\}
    \end{array} \right. 
\end{eqnarray}

and obtain the following rigidity result
\begin{theoremletter}\label{TeoA}
Let \( M = I \times_{\rho} N \) be a warped product manifold, where \( N \) is an \((n-1)\)-dimensional Riemannian manifold, $I= [0,R)$ is an interval, and  \( \rho' > 0 \) in $I$. Let \(\Sigma\) be a convex cone in \( M \) such that \(\Sigma \setminus \{O\}\) is smooth, and let \(\Omega \subset \Sigma\) be a sector-like domain. Assume that there exists a classical solution 
\[
u \in C^{1}(\Omega \cup \Gamma \cup \Gamma_1 \setminus \{O\}) \cap W^{1,\infty}(\Omega) \cap W^{2,2}(\Omega)
\]
to the problem \eqref{SP_cone1}. If \( \operatorname{Ric} \geq (n-1)k g \) and 
\begin{equation}
    \int_{\Omega} u^2 \left[ \varphi \left( kn(n-1) - R \right) - \frac{1}{2} X(R) \right] dv \geq 0,
    \label{hypothesis}
\end{equation}
then \(\Omega = \Sigma \cap B_{r}(x_0)\), where \( B_{r}(x_0) \) is a metric ball of radius \( r \) centered at \( x_0 \). Here, \( R \) is the scalar curvature of \( M \), \( X = \rho \partial_t \), and \( \varphi = \rho' \).
\end{theoremletter}

A key aspect of demonstrating this result relies on the Pohozaev-type identity stated in Lemma \ref{poho_type}, which is a necessary step to apply the \(P\)-function method and  determine the appropriate condition \eqref{hypothesis}. In particular, we derive the following corollary concerning Einstein warped products, which extends the previous analysis of cones in space forms presented in \cite{CiraoloRoncoroni2020}:

\begin{corolario}\label{coro_Einstein}
Let \( M = I \times_{\rho} N \) be a warped product manifold, where \( N \) is an \((n-1)\)-dimensional Riemannian manifold and $I= [0,\infty)$ is an interval. Suppose that $\rho'$ has a sign, and \( M \) is Einstein with \(\operatorname{Ric} = (n-1)k g\). Let \(\Sigma\) be a convex cone in \( M \) such that \(\Sigma \setminus \{O\}\) is smooth, and suppose that \(\Omega \subset \Sigma\) is a sector-like domain for which there exists a solution \( u \in C^{1}(\Omega \cup \Gamma \cup \Gamma_1 \setminus \{O\}) \cap W^{1,\infty}(\Omega) \cap W^{2,2}(\Omega) \) to the problem \eqref{SP_cone1}. Then \(\Omega = \Sigma \cap B_{r}(x_0)\), where \( B_{r}(x_0) \) is a metric ball of radius \( r \) centered at \( x_0 \).  
\end{corolario}

In continuation with our results, it is well known that the overdetermined problem \eqref{firstSP} is closely related to the characterization of compact constant mean curvature (CMC) surfaces without boundary. In this direction, the celebrated Alexandrov’s theorem (see \cite{Alexandrov1}) establishes that the only embedded compact CMC surfaces are spheres. While the method of moving planes was firstly used to prove this result—what subsequently inspired Serrin's original approach in \cite{Serrin1971}—an alternative proof by Reilly and Ros (see \cite{Reilly1977} and \cite{ros}) employs integral methods to study a related Dirichlet problem in a domain whose boundary is the given CMC surface. This integral approach has proven to be particularly interesting for extending such results to more general Riemannian manifolds (see, for example, \cite{brendle}, \cite{Qiu&Xia}, \cite{Fogagnolo}, \cite{hizagi}, and \cite{jia}, \cite{Li&Xia}) and also for sector-like domains in the own Euclidean space. In this last direction, \cite{Pacella2020} characterizes constant mean curvature compact surfaces with boundary which satisfy a gluing condition with respect to the cone $\Sigma$, proving that if either the cone is convex or the surface is a radial graph then $\Gamma$ must be a spherical cap. We highlight that the literature about such problems in cone type domains is vast and have very recent developments. In \cite{lions}, the study of isoperimetric inequalities in convex cones was initiated. The work \cite{Cabre} addresses isoperimetric inequalities in convex cones with homogeneous weights. In \cite{choe}, soap bubble-type results in convex cones are explored. Finally, \cite{poggesi} provides both soap bubble-type and Serrin-type results in possibly non-smooth cones.

In this work, we also contribute to the establishment of a Heintze-Karcher inequality and a soap bubble result in sector-like domains within warped products. In this setting, we consider domains that satisfy the following problem:

\begin{eqnarray} \label{SPSB}
     \left\lbrace \begin{array}{ccccc}
	\Delta u + nku  & = & -1 & \textrm{ in } & \Omega \\
	u & > & 0 & \textrm{ in } & \Omega \\
    u & = & 0 & \textrm{ on } & \Gamma \\
        u_{\nu} & = & 0 & \textrm{ on } & \Gamma_1 \setminus \{O\}
    \end{array} \right.
\end{eqnarray}

In the following result, we consider the value 
\begin{equation}\label{c_value}
c \doteq \frac{1}{Area(\Gamma)}\left(Vol(\Omega)+nk\int_{\Omega}u\,dv\right)
\end{equation}
While, in general, it is expected that \( c \) depends on the solution \( u \), the case \( k = 0 \) is particularly meaningful and provides a parallel with \cite[Theorem 2.6]{poggesi} (see Remark \ref{rem_k=0}).

\begin{theoremletter}\label{TeoB}
Let \( M = I \times_{\rho} N \) be a warped product manifold, where \( N \) is an \((n-1)\)-dimensional Riemannian manifold and $I=[0,\infty)$ is an interval. Let \(\Sigma\) be a convex cone in \( M \) such that \(\Sigma \setminus \{O\}\) is smooth, and let \(\Omega \subset \Sigma\) be a sector-like domain such that exists a solution $u \in C^{1}(\Omega \cup \Gamma \cup \Gamma_1\setminus\{O\} ) \cap W^{1,\infty}(\Omega) \cap W^{2,2}(\Omega)$ to the problem \eqref{SPSB}. If $Ric \geq (n-1)kg,$ for some $k \in \Rset$, then  
    \begin{eqnarray*}
        \int_{\Gamma}^{} (H_0 - H)(u_{\nu})^2 da\geq 0,
    \end{eqnarray*}
    where $H$ is the mean curvature of $\Gamma$ and $H_0 = \frac{n-1}{nc}$ with $c$ constant given by \eqref{c_value}. In particular, if $H \geq H_0$ on $\Gamma$ then $\Omega = \Sigma \cap B_{r}(x_0)$ where $B_{r}(x_0)$ is a metric ball.
\end{theoremletter}

The last part of this paper analyses another rigidity result in the context of sector-like domains, related to an overdetermined problem for the drift Laplacian, as follows. Recall that if $f$ is a smooth positive function, the {\it drift Laplacian} (or {\it weighted Laplacian}) is the operator defined by
\[
\Delta_{f} \doteq \Delta + \langle \nabla \log f, \nabla \cdot \rangle.
\]
In \cite{Ruan2018}, the authors obtained a Serrin-type result related to the problem \eqref{firstSP}, where $\Omega \subset \mathbb{R}^{n}$ is a Euclidean domain and the Laplacian is replaced by the above-mentioned operator. Furthermore, $f$ is a homogeneous weight of degree \(\alpha > 0\), which means that \(\langle \nabla f, x \rangle = \alpha f\), where \(x\) is the position vector. In this context, under both Dirichlet and Neumann constant boundary conditions, and with the weight function $f$ being a homogeneous function, the domain $\Omega$ is necessarily a ball. As in the classical problem, such overdetermined problems are related to Alexandrov-type results, but now considering the weighted mean curvature, where the hypothesis of a homogeneous weight is commonly used (see \cite{batista} and \cite{CC_2014}). See also \cite{AFS2024} for some analogous results in generalized cones.

Motivated by the findings in \cite{Ruan2018}, we study the overdetermined problem
\begin{eqnarray} \label{SPRuan1}
    \left\lbrace \begin{array}{ccccc}
	\Delta_f u  & = & -1 & \textrm{ in } & \Omega \\
	u & = & 0 & \textrm{ on } & \Gamma \\
        u_{\nu} & = & -c & \textrm{ on } & \Gamma \\
	    u_{\nu} & = & 0 & \textrm{ on } & \Gamma_1 \setminus \{O\} \\
        \nabla^2 \log f (\nabla u, \nabla u) + \displaystyle\frac{\< \nabla \log f, \nabla u \>^2 }{\alpha} &\leq & 0 & \textrm{ in } & \Omega
    \end{array} \right. 
\end{eqnarray}
in a sector-like domain $\Omega$ contained in a convex cone $\Sigma\subset\mathbb{R}^n$ and obtaining the following result. 

\begin{theoremletter}\label{teoC}
     Let $\Sigma$ a convex open cone in $\Rset^n$ such that $\Sigma \setminus \{O\}$ is smooth and $\Omega \subset \Sigma$ a sector-like domain. If $f$ is a homogeneous weight of degree \(\alpha > 0\) and there exists a solution $u \in C^{1}(\Omega \cup \Gamma \cup \Gamma_1\setminus\{O\} ) \cap W^{1,\infty}(\Omega) \cap W^{2,2}(\Omega)$ to the problem \eqref{SPRuan1}, then $\Omega = \Sigma \cap B_{r}(x_0)$,where $B_{r}(x_0)$ is a ball, and $u=r^2-\frac{|x|^2}{2(n+\alpha)}.$
\end{theoremletter}

\textbf{Overview of the paper:} In Section 2, we recall the basic definitions and preliminary results necessary to prove Theorem \ref{TeoA} and its consequences, particularly establishing a Pohozaev-type identity in the context of cones. In Section 3, we derive a Soap Bubble result and a Heintze-Karcher inequality related to the problem, particularly proving Theorem \ref{TeoB}. Finally, in Section 4, we study the drift Laplacian to obtain a rigidity characterization for an overdetermined problem for cones in the Euclidean space.
\section{Rigidity for a Serrin's type problem in convex cones}

Throughout this section, we consider a warped product $M^n = I \times_{\rho} N^{n-1}$ equipped with the metric  
\[
g = dt^2 + \rho^2(t) g_N,
\]  
where $I= [0, \infty)$ is an interval, $(N, g_N)$ is an $(n-1)$-dimensional Riemannian manifold, and $\rho : I \to \mathbb{R}$ is a positive smooth function. 

We denote by $\Sigma$ the open cone with vertex at $\{O\}$, defined as  
\[
\Sigma \doteq \{t x \mid x \in \omega,\, t \in I\},  
\]  
for some open domain $\omega \subset N$. Furthermore, $\Omega$ is a sector-like domain with boundary $\partial \Omega = \Gamma \cup \Gamma_1$, as in Definition \ref{def_sec} (see also Figure \ref{sector-like}). As before, we consider the following problem:  

\begin{eqnarray} \label{SP}
    \left\lbrace \begin{array}{ccccc}
	\Delta u + n k u  & = & -1 & \textrm{ in } & \Omega, \\
	u & > & 0 & \textrm{ in } & \Omega,
    \end{array} \right. 
\end{eqnarray}
with boundary conditions  
\begin{eqnarray} \label{SPBC}
    \left\lbrace \begin{array}{ccccc}
	u & = & 0 & \textrm{ on } & \Gamma, \\
        u_{\nu} & = & -c & \textrm{ on } & \Gamma, \\
        u_{\nu} & = & 0 & \textrm{ on } & \Gamma_1 \setminus \{O\}.
    \end{array} \right. 
\end{eqnarray}

\begin{observacao}
In order to apply the divergence theorem and investigate the validity of maximum principles in sector-like domains (where there is a lack of regularity at $O$ and along $\partial \Gamma$), we assume in our results that  
\[
u \in C^1(\Omega \cup \Gamma \cup \Gamma_1 \setminus \{O\}) \cap W^{1,\infty}(\Omega) \cap W^{2,2}(\Omega).
\]
See \cite[Lemma 2.1, Corollary 2.3]{Pacella2020} and \cite[Lemma 6]{CiraoloRoncoroni2020} for the approximation arguments, which are also valid in our situation.
\end{observacao}

Our first result, which is of independent interest, is a Pohozaev-type formula associated with the problem \eqref{SP}-\eqref{SPBC}. While the nature of this formula is similar to that of \cite[Lemma 3]{FAM_2023} and \cite[Proposition 7]{FR_2022}, we provide a detailed proof to emphasize the striking resemblance between the derivation of this formula and the genesis of this new type of domain and boundary conditions. This approach clarifies the importance of the coupled overdetermined conditions on the boundaries $\Gamma$ and $\Gamma_{0}$. As in \cite{FR_2022}, we also emphasize (and explicitly denote in the formula) the key property required to obtain an identity of this type: the existence of a closed conformal vector field on $M$. In this context, $X = \rho \partial_t$ is a vector field satisfying  
\[
\nabla_Y X = \varphi Y,
\]  
for all vector field $Y$ in $M$, where $\varphi = \rho'$ is known as the conformal factor. We adopt this notation from now on.
\begin{lema}\label{poho_type}
Let \( M = I \times_{\rho} N \) be a warped product manifold, where \( N \) is an \((n-1)\)-dimensional Riemannian manifold, $I= [0,\infty)$. If $u$ is a solution of \eqref{SP}-\eqref{SPBC}, then
 \begin{eqnarray*}
       c^2 \int_{\Omega}^{} \rho'\;dv = \Big( \frac{n+2}{n} \Big) \int_{\Omega}^{} \rho' u \;dv + 2k \int_{\Omega}^{} \rho' u^2 \;dv +  \frac{n-2}{2n(n-1)} \int_{\Omega}^{} \Big(\rho' R+\frac{X(R)}{2}\Big)u^2  \;dv,
    \end{eqnarray*}   
    where $X=\rho\partial_{t}$ and $R$ is the scalar curvature of $M$.
\end{lema}

\begin{proof}

We recall that since $X=\rho \partial_{t}$ is a closed conformal vector field, then $\Div X=n\varphi$, where $X=\rho'$. A straightforward calculation shows that
 \begin{eqnarray*}
\Div \Big( \frac{|\nabla u|^2}{2}X - \< X, \nabla u \> \nabla u \Big)\!\!\!\!&=&\!\!\!\! \frac{|\nabla u|^2}{2} \Div(X) + \< \nabla \Big( \frac{|\nabla u|^2}{2} \Big), X \> - \< X, \nabla u \> \Delta u - \< \nabla \< X, \nabla u \>, \nabla u \> \\
     &=& \Big( \frac{n-2}{2} \Big) \varphi |\nabla u|^2 - \< X, \nabla u \> \Delta u.
 \end{eqnarray*}

  Since  $u_\nu=\< X , \nu \> = 0$ on $\Gamma_1 \setminus \{O\}$ it follows from the divergence theorem that 
    \begin{eqnarray*}\label{diverpoh}
        \int_{\Omega}^{} \Div \Big( \frac{|\nabla u|^2}{2}X - \< X, \nabla u \> \nabla u \Big) \;dv 
        &=& - \frac{c^2}{2} \int_{\partial\Omega}^{} \< X, \nu \> \;da \\
        &=& - \frac{c^2n}{2} \int_{\Omega}\varphi\;dv\nonumber
    \end{eqnarray*}

On the other hand, note that
\begin{eqnarray*}
        \int_{\Omega}^{} \Div(\varphi u \nabla u) \;dv &=& \int_{\Omega}^{} \varphi u \Delta u \;dv + \int_{\Omega}^{} u \< \nabla u, \nabla \varphi \> \;dv + \int_{\Omega}^{} \varphi |\nabla u|^2 \;dv
    \end{eqnarray*}

Thus, from the overdetermined conditions \eqref{SPBC}  we have that 

\begin{equation}
        \label{eq004} 
        \int_{\Omega}^{} \varphi |\nabla u|^2 \;dv= - \int_{\Omega}^{} \Big[ \varphi u \Delta u + u \< \nabla u, \nabla \varphi \> \Big] \;dv.
    \end{equation}

    
From \eqref{eq004} and \eqref{diverpoh} we obtain
\begin{equation}\label{int0}
      \frac{c^2n}{2} \int_{\Omega}\varphi\;dv=\frac{n-2}{2}\int_{\Omega}^{} \Big[ \varphi u \Delta u + u \< \nabla u, \nabla \varphi \> \Big] \;dv+\int_\Omega\langle X,\nabla u\rangle\Delta u\;dv
\end{equation}

Now, let us calculate the right side of the above equation as follows. Firstly, since $u$ is a solution of the Serrin problem we get that 

$$\int_{\Omega}\varphi u\Delta u\;dv=-\int_{\Omega}\varphi u\;dv-nk\int_\Omega\varphi u^2\;dv.$$
Morever, from the overdeterminated conditions, the divergence theorem provides that

\begin{eqnarray}\label{int}
\int_\Omega\langle X,\nabla u\rangle\Delta u\;dv&=&-\int_\Omega \langle X,\nabla u\rangle\;dv-nk\int_{\Omega}u\langle X,\nabla u\rangle\;dv \nonumber \\
&=&-\int_{\Omega}(div(uX)-u\,div X)\;dv-nk \int_{\Omega}u\langle X,\nabla u\rangle\;dv \nonumber \\
&=&n\int_{\Omega}u\varphi\;dv-\frac{nk}{2}\int_\Omega(div(u^2X)-n\varphi u^2)\;dv \nonumber \\
&=&n\int_{\Omega}u\varphi\;dv+\frac{n^2k}{2}\int_\Omega \varphi u^2\;dv
\end{eqnarray}

    Plugging \eqref{int} and \eqref{int0} into \eqref{eq004} we conclude that
    \begin{eqnarray}
       c^2 \int_{\Omega}^{} \varphi \;dv = \Big( \frac{n+2}{n} \Big) \int_{\Omega}^{} \varphi u \;dv + 2k \int_{\Omega}^{} \varphi u^2 \;dv + \Big( \frac{n-2}{n} \Big) \int_{\Omega}^{} u \< \nabla \varphi, \nabla u \> \;dv
        \label{eq009}
    \end{eqnarray}

Now, from our overdetermined conditions we get that

\begin{equation}\label{b}
    \int_{\Omega}u\langle\nabla \varphi,\nabla u\rangle\;dv=-\frac{1}{2}\int_{\Omega}u^2\Delta\varphi\;dv.
\end{equation}


Finally, since $Ric(X)=-(n-1)\nabla\varphi,$ we conclude that
\begin{eqnarray}\label{trick2}
 -(n-1)\Delta\varphi&=& div(Ric(X)) \nonumber \\
    &=& \frac{1}{2}\langle Ric, \mathcal{L}_Xg\rangle+div(Ric)(X) \nonumber \\
    &=&\varphi R+\frac{1}{2}X(R),
    \end{eqnarray}
where $R$ denotes the scalar curvature of $\Omega.$
From \eqref{eq009}, \eqref{b} and \eqref{trick2} we get the desired result.
\end{proof}

\begin{observacao}
While the core of the last result is the property that $M$ possesses a closed conformal vector field, we recall that manifolds admitting a nontrivial closed conformal vector field are locally isometric to a warped product with a 1-dimensional factor. For details, we refer the reader to \cite[Section~3]{montiel}.
\end{observacao}




We also recall the following Obata-type result that characterizes the rigidity of the Serrin-type problem. We refer to \cite[Lemma 6]{FR_2022} and Step 4 in \cite[Proof of Theorem 2]{CiraoloRoncoroni2020} for this result.

\begin{lema} \label{rigidez_Tipo_Obata}
    Let $(M^n, g)$ be an n-dimensional Riemannian manifold (not necessarily complete). Take $\Sigma$ a open  cone in $M$ such that $\Sigma \setminus \{O\}$ is smooth, $\Omega \subset \Sigma$ is a sector-like domain, so that $\overline{\Omega}$ is compact, and there exists a classical solution $u \in C^{1}(\Omega \cup \Gamma \cup \Gamma_1\setminus\{O\} ) \cap W^{1,\infty}(\Omega) \cap W^{2,2}(\Omega)$ to the problem 
    \begin{eqnarray*} 
         \left\lbrace \begin{array}{ccccc}
	\nabla^2 u & = & \Big( -\frac{1}{n} - ku \Big) g & \textrm{       in } & \Omega \\
	u & > & 0 &  &  \\
	u & = & 0 & \textrm{ on } & \Gamma \\
    	    u_{\nu} & =  & 0 & \textrm{ on } & \Gamma_1 \setminus \{O\},
        \end{array} \right. 
    \end{eqnarray*}
    where $k \in \Rset.$ Then $\Omega = \Sigma \cap B_{r}(x_0)$ where $B_r(x_0)$ is a metric ball of radius $r$ centered at $x_0$ and $u$ is a radial function, i.e. u depends only on the distance from the center of the ball.
\end{lema}

Now, we prove the main result of this section, Theorem \ref{TeoA}. While the initial steps of this proof are similar to those presented in previous works, we highlight the interesting hypothesis introduced here. This hypothesis is motivated by a application of the $P$-function method coupled with Lemma \ref{poho_type}. For the sake of completeness, we will present the full proof, including all the steps.

 \begin{demonstracao}[Proof of Theorem \ref{TeoA}]
 Let $u$ be a solution to the problem \eqref{SP_cone1} and consider the $P$-function defined by 
     \begin{equation*}
         P = |\nabla u|^2 + \frac{2}{n}u + ku^2.
     \end{equation*}
    A classical computation that comes from Bochner's formula shows that since $Ric\geq (n-1)k g$, then 
    \begin{equation}\label{sub_harmonic}
     \Delta P(u)\geq 2\left|\nabla^2 u-\frac{\Delta u}{n}g\right|^2,  
    \end{equation} 
    this is $P$ is a sub-harmonic function (see, for example \cite[Lemma 5]{FR_2022}). Furthermore, since $u=0$ on $\Gamma$ and from \eqref{SPBC}, we have that $P=c^2$ on $\Gamma.$
    Moreover,
    \begin{equation}
        \label{eq001} 
        \nabla P = 2 \nabla^2 u (\nabla u,\cdot) + \frac{2}{n} \nabla u + 2k u\nabla u
    \end{equation}
    and from the convexity assumption of the cone $\Sigma,$ we have, on $\Gamma_{1}$,
    \begin{eqnarray}
         0 &=& g(\nabla \< \nabla u, \nu \>, \nabla u)  \nonumber \\
         &=& \nabla^2u (\nabla u, \nu) + II(\nabla u, \nabla u) \nonumber \\
         &\geq& \nabla^2u (\nabla u, \nu) \textrm{ on } \Gamma_1. \label{eq002}
    \end{eqnarray}
    From \eqref{eq001} and \eqref{eq002}, and recalling that $u_{\nu}=0$ on $\Gamma_{1}$, we obtain
    \begin{eqnarray*}
       P_{\nu} &\doteq& \< \nabla P ,\nu \> \\
       &=& 2 \nabla^2u (\nabla u, \nu) \leq 0 \textrm{ on } \Gamma_1\setminus \{O\}.
    \end{eqnarray*}
    Hence, the P-function satisfies
    \begin{eqnarray*} 
    \left\lbrace \begin{array}{ccccc}
	\Delta P & \geq & 0 & \textrm{ in } & \Omega \\
	P & = & c^2 & \textrm{ on } & \Gamma \\
	    P_{\nu} & \leq & 0 & \textrm{ on } & \Gamma_1 \setminus \{O\}. \\
    \end{array} \right. 
    \end{eqnarray*}

We have $P \leq c^2$ in $\Omega.$ Indeed, multiplying $\Delta P \geq 0$ by $(P-c^2)^{+}$ and by integration by parts we obtain 
    \begin{eqnarray*}
        0 \geq \int_{\Omega \cap \{ P > c^2\}}^{} |\nabla P|^2 \;dv- \int_{\partial \Omega}^{} (P - c^2)^{+} P_{\nu} \;da. 
    \end{eqnarray*}
    Since $P=c^2$ on $\Gamma$ and $P_{\nu} \leq 0$ on $\Gamma_1$ we obtain that
\begin{eqnarray*}
0 \geq \int_{\Omega \cap \{ P > c^2 \} }^{} |\nabla P|^2 \;dv \geq 0
\end{eqnarray*}
and them $\nabla P = 0$ on $\Omega \cap \{ P > c^2 \}.$ Follows $P$ is a constant function on $\Omega \cap \{ P > c^2 \},$ since $P =c^2$ on $\Gamma,$ by continuity, $P \equiv c^2$ on $\overline{\Omega}.$ In this case, $\{ P > c^2 \}$ is a empty set and we have $P \leq c^2$ in $\Omega.$ 

Now, let us prove that $P=c^2$ in $\overline{\Omega}.$ Indeed, suppose by contradiction that $P < c^2$ in $\Omega.$ 
   
    Since $\varphi > 0$ and $P < c^2,$ we have
    \begin{eqnarray*}
        \varphi c^2 > \varphi \Big( |\nabla u|^2 + \frac{2}{n}u + ku^2 \Big),
    \end{eqnarray*}
    then by \eqref{SP} and \eqref{eq004}
    \begin{eqnarray}
        c^2 \int_{\Omega}^{} \varphi \;dv &>& \int_{\Omega}^{} \varphi |\nabla u|^2 \;dv + \frac{2}{n} \int_{\Omega}^{} \varphi u \;dv + k \int_{\Omega}^{} \varphi u^2 \;dv \nonumber \\
        &=& -\int_{\Omega}^{} \varphi u \Delta u \;dv - \int_{\Omega}^{} u \< \nabla \varphi, \nabla u \> \;dv + \frac{2}{n} \int_{\Omega}^{} \varphi u \;dv + k \int_{\Omega}^{} \varphi u^2 \;dv \nonumber\\
        &=& \int_{\Omega}^{} \varphi u \Big( 1 + nku \Big) \;dv  - \int_{\Omega}^{} u \< \nabla \varphi, \nabla u \> \;dv + \frac{2}{n} \int_{\Omega}^{} \varphi u \;dv + k \int_{\Omega}^{} \varphi u^2 \;dv  \nonumber\\
        &=& \Big(1 + \frac{2}{n} \Big) \int_{\Omega}^{} \varphi u \;dv + k(n+1) \int_{\Omega}^{} \varphi u^2 \;dv - \int_{\Omega}^{} u \< \nabla \varphi, \nabla u \> \;dv. \label{eq005} 
    \end{eqnarray}

    Putting together \eqref{eq005} and Lemma \ref{poho_type} we have 
    \begin{eqnarray*}
        0 > k(n-1) \int_{\Omega}^{} \varphi u^2 \;dv- \frac{1}{n} \int_{\Omega}^{} u^2\Big[\varphi R + \frac{1}{2} X(R)  \Big]\;dv,
    \end{eqnarray*}
    which is a contradiction in view of \eqref{hypothesis}.
    
    Since $P$ is constant, then $\Delta P = 0$ 
    and from \eqref{sub_harmonic} we have that $u$ satisfies
    the following Obata-type problem 
\begin{eqnarray*}
\left\lbrace \begin{array}{ccccc}
\nabla^2 u & = & \Big( -\frac{1}{n} - ku \Big) g. & \textrm{ in } & \Omega \\
u & > & 0 &  &  \\
u & = & 0 & \textrm{ on } & \Gamma \\
u_{\nu} & =  & 0 & \textrm{ on } & \Gamma_1 \setminus \{O\}
\end{array} \right. 
\end{eqnarray*}
and the result follows from Lemma \ref{rigidez_Tipo_Obata}. 
\end{demonstracao}
\begin{observacao}\label{rem_sign}
As observed in the proof, it is essential that the conformal factor \(\rho'\) be positive in order to obtain \eqref{eq005}. Additionally, a consequence of the same analysis is that if \(\rho' < 0\), an analogous result holds by requiring \eqref{hypothesis} with the reverse inequality.     
\end{observacao}

When the warped product $M=I\times_{\rho} N$ is Einstein we have the interesting consequence, the Corollary \ref{coro_Einstein}:
\vspace{0.2cm}

\begin{demonstracao}[Proof of Corollary \ref{coro_Einstein}]
Since \(\operatorname{Ric} = (n-1)k g\), the (constant) scalar curvature is \(R = n(n-1)k\). Since \(X(R) = 0\) in this case, we conclude that the equality in \eqref{hypothesis} is satisfied. Thus, independently of the sign of \(\rho'\) (see Remark \ref{rem_sign}), we conclude the result.
\end{demonstracao}

Such characterizations have a stronger consequence in a particular case of model manifolds: beyond characterizing geodesic balls, we can also characterize the ambient space \(M\). As in \cite{Roncoroni2018}, we recall that a Riemannian manifold \((M^{n}_{\rho}, g_{M^{n}_{\rho}})\) is called a model manifold if:
\begin{equation*}
M^{n}_{\rho}\doteq\dfrac{[0,R) \times \mathbb{S}^{n-1}}{\sim} \quad {and} \quad g_{M^{n}_{\rho}}:=dt\otimes dt + \rho^{2}(t) g_{\mathbb{S}^{n-1}};
\end{equation*}
where $R\in(0,+\infty]$, $\sim$ is the relation that identifies all the points of $\lbrace 0\rbrace\times\mathbb{S}^{n-1}$ and $\rho:[0,R)\rightarrow[0,+\infty)$ is a smooth function such that:
\begin{itemize}
\item $\rho(t) >0$, for all $t>0$;
\item $\rho^{(2k)}(0)=0$, for all $k= 0,1,2,\dots$;
\item  $\rho'(0)=1$.
\end{itemize}
Moreover, the unique point corresponding to \(t=0\) is called the pole of the model and is denoted by \(o \in M\), and \(\rho\) is called the warping function. We highlight that, as mentioned in the introduction, the space forms are examples of such structures. The following result is a version of \cite[Theorem 4]{FR_2022} in this setting.

\begin{corolario}
	Let $(M,g)$ be a model manifold such that $\operatorname{Ric} \geq (n-1)k g$ and $\rho'>0$. Let $\Sigma$ be a convex open cone in $M$ such that $\Sigma\setminus\left\{O\right\}$ is smooth, and let $\Omega\subset \Sigma$ be a sector-like domain domain. Assume that there exists a classical solution 
\[
u \in C^{1}(\Omega \cup \Gamma \cup \Gamma_1 \setminus \{O\}) \cap W^{1,\infty}(\Omega) \cap W^{2,2}(\Omega)
\]
to the problem \eqref{SP_cone1}. If $u$ satisfies the compatibility condition \eqref{hypothesis}, then $\Omega = \Sigma \cap B_{r}(O)$ where $B_{r}(O)$ is a metric ball centred at $O$ of radius $r$ and $u$ is a radial function given by
	\begin{equation*}
	u(t)=\begin{cases}
	\frac{\cosh(\sqrt{-k}t)}{kn\cosh(\sqrt{-k}r)}-\frac{1}{nk} &\mbox{if } k<0\, , \\ \frac{r^2}{2n}-\frac{t^2}{2n} &\mbox{if } k=0, \\
	\frac{\cos(\sqrt{k}t)}{kn\cos(\sqrt{k}r)}-\frac{1}{nk} &\mbox{if } k>0 \, ,
	\end{cases}
	\end{equation*}
	where $t$ is the geodesic distance from $O$.
	
	Moreover, the warping function $\sigma$ in the metric ball is given by the following expression:
	\begin{equation*}
	\sigma(t)=\begin{cases}
	 \frac{\sinh(\sqrt{-k}t)}{\sqrt{-k}} &\mbox{if } k<0, \\ 
	 t &\mbox{if } k=0, \\ 
	 \frac{\sin(\sqrt{k}t)}{\sqrt{k}}&\mbox{if } k>0\, .
	\end{cases}
	\end{equation*}
In particular, $(M,g)$ is a space form.    
\end{corolario}
Another consequence is the following result, which holds in the case \(k=0\). In this setting, we assume the hypothesis of constant scalar curvature.
\begin{corolario} 
Let $(M,g)$ be a model manifold that has $\operatorname{Ric} \geq 0$, $\rho'>0$ and constant scalar curvature $R$. Let $\Sigma$ be a convex open cone in $M$ such that $\Sigma\setminus\left\{O\right\}$ is smooth, and let $\Omega\subset \Sigma$ be a sector-like domain domain. Assume that there exists a classical solution 
\[
u \in C^{1}(\Omega \cup \Gamma \cup \Gamma_1 \setminus \{O\}) \cap W^{1,\infty}(\Omega) \cap W^{2,2}(\Omega)
\]
to the problem \eqref{SP_cone1} with $k=0$. Then $\Omega = \Sigma \cap B_r(O)$ where $B_r(O)$ is the Euclidean ball of radius $r$ centred in the pole $O$ and $u$ is a radial function. 
\end{corolario}

\begin{demonstracao}
Since \(R\) is constant, \(\rho' > 0\), and \(R \geq 0\), the compatibility condition \eqref{hypothesis} is satisfied. Furthermore, since the \(P\)-function is constant, the inequality obtained in \eqref{sub_harmonic} implies that the Ricci curvature satisfies \( \operatorname{Ric}(\nabla u, \nabla u) = 0 \). Given the explicit expression fro the Ricci curvature in this case, \( \operatorname{Ric} = -(n-1) \frac{\rho''}{\rho} \), it follows that \(\rho'' = 0\) and then \(\rho(t) = t\). Hence, $\Omega = \Sigma \cap B_r(O)$, where \(B_r(O)\) is a ball centred in $O$ and $M$ is the Euclidean space.
\end{demonstracao}

\section{Soap Bubble and Heintze-Karcher inequality in cones} 

In this section, we describe a Soap Bubble result and a Heintze-Karcher inequality in the setting of sector-like domains in Riemannian manifolds as described earlier. A crucial tool for this purpose is the well-known Reilly's identity, proved in \cite{Reilly1977}: 
\begin{eqnarray}
    \int_{\Omega} \Big( \frac{n-1}{n}(\Delta f)^2 - |\mathring{\nabla}^2 f|^2 \Big) \;dv= \int_{\partial \Omega} \Big( h(\overline{\nabla} z, \overline{\nabla} z) + 2 f_{\nu} \overline{\Delta} z + H f_{\nu}^2 \Big) \;da+ \int_{\Omega} Ric(\nabla f, \nabla f)\;dv, \label{reilly}
\end{eqnarray}
which holds true for every domain $\Omega$ in a Riemannian manifold $(M^n, g)$ and for every $f \in C^{\infty}(\Omega)$. Here, $\mathring{\nabla}^2 f$ denotes the traceless Hessian of $f$, explicitly given by 
\[
\mathring{\nabla}^2 f = \nabla^2 f - \frac{\Delta f}{n} g,
\] 
and $\overline{\nabla}$ and $\overline{\Delta}$ indicate the gradient and the Laplacian of the induced metric on $\partial \Omega$, respectively. Additionally, $z = f|_{\partial \Omega}$, $\nu$ is the unit outward normal to $\partial \Omega$, $h(X, Y) = g(\nabla_X \nu, Y)$ is the second fundamental form, and $H = \textrm{tr}_g h$ is the mean curvature (with respect to $-\nu$) of $\partial \Omega$.
 \\

During this section we study sector-like domains $\Omega$ as described in Definition \ref{def_sec} such that exists a solution $u \in C^{1}(\Omega \cup \Gamma \cup \Gamma_1\setminus\{O\} ) \cap W^{1,\infty}(\Omega) \cap W^{2,2}(\Omega)$ to the following problem
\begin{eqnarray}\label{eq_soap} 
     \left\lbrace \begin{array}{ccccc}
	\Delta u + nku  & = & -1 & \textrm{ in } & \Omega \\
	u & > & 0 & \textrm{ in } & \Omega \\
    u & = & 0 & \textrm{ on } & \Gamma \\
        u_{\nu} & = & 0 & \textrm{ on } & \Gamma_1 \setminus \{O\}
    \end{array} \right.
\end{eqnarray}
Recall that, in this situation, $\partial\Omega=\Gamma\cup\Gamma_{1}$, where $\Gamma\subset\Sigma$ and $\Gamma_{1}\subset\partial\Sigma$, where $\Sigma$ is a cone in $M$ such that $\Sigma \setminus \{O\}$ is smooth.

\begin{lema} \label{LemmaReilly}
    Let $u$ be a solution of \eqref{eq_soap}. Then
    \begin{eqnarray*}
        \int_{\Omega}^{} |\mathring{\nabla}^2 u|^2 \;dv+ \int_{\Omega}^{} \left[Ric - (n-1)kg \right](\nabla u, \nabla u) \;dv&=& - \frac{1}{n} \int_{\Gamma}^{} u_{\nu}\left[ (n-1) +nHu_{\nu} \right]\;da \\\nonumber
        &-& \int_{\Gamma_1}^{} h(\overline{\nabla} z, \overline{\nabla} z)\;da. 
    \end{eqnarray*}
\end{lema}
\begin{demonstracao}
    Since $u=0$ on $\Gamma$ and $u_{\nu}=0$ on $\Gamma_{1}$, Reilly's identity \eqref{reilly} applied to the solution $u$ of \eqref{eq_soap} implies 
    \begin{eqnarray}\label{eq_lemma}
        \int_{\Omega}^{} |\mathring{\nabla}^2 u|^2\;dv &=& - \int_{\Gamma}^{}  H u_{\nu}^{2}\;da - \int_{\Gamma_1}^{}  h(\overline{\nabla} z, \overline{\nabla} z)\;da-\int_{\Omega}^{} Ric(\nabla u, \nabla u)\;dv \\\nonumber
        &+& \frac{n-1}{n} \int_{\Omega}^{} (\Delta u)^2\;dv. 
    \end{eqnarray}
Now, using again \eqref{eq_soap} and the divergence theorem, we have
 \begin{eqnarray*}
  \int_{\Omega}^{} (\Delta u)^2 \;dv &=&-\int_{\Omega}\Delta u\;dv-nk\int_{\Omega}u\Delta u\;dv\\
  &=&-\int_{\Gamma\cup\Gamma_{1}}u_{\nu}\;da-nk\left(\int_{\Gamma\cup\Gamma_{1}}u u_{\nu}\;da-\int_{\Omega}|\nabla u|^2\;dv\right)\\
  &=&-\int_{\Gamma} u_{\nu}\;da+nk\int_{\Omega}|\nabla u|^2\;dv.
 \end{eqnarray*}
 Replacing this last information in \eqref{eq_lemma}, we conclude the result.
\end{demonstracao}

The key observation is that a hypothesis of Ricci curvature bounded below, coupled with the last lemma, establishes that the vanishing of $\mathring{\nabla}^2 u$ (and, therefore, the rigidity of sector-like domains) is controlled by the right-hand side of \eqref{eq_lemma}. That is, the rigidity depends only on the boundary. This is the essence of the next results. First, we present an immediate consequence.

\begin{corolario}
Let \( M = I \times_{\rho} N \) be a warped product manifold, where \( N \) is an \((n-1)\)-dimensional Riemannian manifold and $I= [0,\infty)$ is an interval. Let \(\Sigma\) be a convex cone in \( M \) such that \(\Sigma \setminus \{O\}\) is smooth, and let \(\Omega \subset \Sigma\) be a sector-like domain such that exists a solution $u \in C^{1}(\Omega \cup \Gamma \cup \Gamma_1\setminus\{O\} ) \cap W^{1,\infty}(\Omega) \cap W^{2,2}(\Omega)$ to the problem \eqref{eq_soap}. If $Ric \geq (n-1)kg,$ for some $k \in \Rset$ and $H = -\frac{1}{u_{\nu}} \frac{n-1}{n}$ on $\Gamma,$ then $\Omega = \Sigma \cap B_{r}(x_0)$ where $B_{r}(x_0)$ is a metric ball ad $u$ is radial function. 
\end{corolario}
\begin{demonstracao}
    Since $\Sigma$ is convex and $Ric \geq (n-1)k g$, we have 
    \begin{eqnarray*}
        \int_{\Omega}^{} |\mathring{\nabla}^2 u|^2\;dv\leq\int_{\Omega}^{} |\mathring{\nabla}^2 u|^2 \;dv+ \int_{\Omega}^{} \left[Ric - (n-1)kg \right](\nabla u, \nabla u) \;dv\leq - \frac{1}{n} \int_{\Gamma}^{} u_{\nu}\left[ (n-1) +nHu_{\nu} \right]\;da.
    \end{eqnarray*}
    From $H = -\frac{1}{u_{\nu}} \cdot \frac{n-1}{n}$, we obtain that $\mathring{\nabla}^2 u = 0$. Therefore, the desired result follows from Lemma \ref{rigidez_Tipo_Obata}.
\end{demonstracao}
As clarified in the last result, we observe that the convexity of $\Sigma$ and the hypothesis $Ric \geq (n-1)k g$ imply  
\begin{eqnarray}\label{traceless}
        \int_{\Omega}^{} |\mathring{\nabla}^2 u|^2 \;dv\leq - \frac{1}{n} \int_{\Gamma}^{} u_{\nu} \left[ (n-1) + nHu_{\nu} \right]\;da.
\end{eqnarray}  
This will be important for the next two results.

\vspace{0.2cm}
 \begin{demonstracao}[Proof of Theorem \ref{TeoB}]
    We note that
    \begin{eqnarray*}
    \int_{\Gamma}^{} H u_{\nu}^{2} \;da&=& H_0 \int_{\Gamma}^{} u_{\nu}^{2} \;da+ \int_{\Gamma}^{} (H-H_0) u_{\nu}^{2} \;da\\
        &=& H_0 \int_{\Gamma}^{} (u_{\nu}+c)^2 \;da- 2c H_0 \int_{\Gamma}^{} u_{\nu} \;da- c^2 H_{0} Area(\Gamma) + \int_{\Gamma}^{} (H-H_0) u_{\nu}^{2}\;da\\
        &=&\frac{n-1}{nc}\int_{\Gamma}^{} (u_{\nu}+c)^2\;da-\frac{2(n-1)}{n}\int_{\Gamma}^{} u_{\nu}\;da-\frac{(n-1)c}{n}Area(\Gamma)+\int_{\Gamma} (H-H_0) u_{\nu}^{2}\;da,
    \end{eqnarray*}
where $H_{0}=\frac{n-1}{nc}$. Then,
\begin{eqnarray}
-\frac{1}{n}\int_{\Gamma}u_{\nu}[(n-1)+nHu_{\nu}]\;da&=&\frac{n-1}{n}\int_{\Gamma}u_{\nu}\;da+\frac{(n-1)c}{n}Area(\Gamma)\nonumber\\
& &-\frac{n-1}{nc}\int_{\Gamma}(u_{\nu}+c)^{2}\;da+\int_{\Gamma} (H_0-H) u_{\nu}^{2} \;da \label{id_rhs}  
\end{eqnarray}
For the constant $c$ defined in \eqref{c_value}, we observe that
    \begin{eqnarray}
        c &=& -\frac{1}{Area(\Gamma)}\int_{\Omega}\Delta u\;dv\nonumber\\
        &=&-  \frac{\int_{\partial\Omega}^{} u_{\nu}\;da}{Area(\Gamma)}\nonumber\\
        &=&-  \frac{\int_{\Gamma}^{} u_{\nu}\;da}{Area(\Gamma)}\label{C-def2}.
    \end{eqnarray}
Replacing \eqref{C-def2} in \eqref{id_rhs} and after in \eqref{traceless} we obtain
\begin{equation}\label{soap_final}
\int_{\Omega}^{} |\mathring{\nabla}^2 u|^2\;dv+\frac{n-1}{nc}\int_{\Gamma}(u_{\nu}+c)^{2}\;da\leq \int_{\Gamma} (H_0-H) u_{\nu}^{2} \;da.  
\end{equation}
Since the two integrals on the left-hand side of \eqref{soap_final} are nonnegative, the inequality follows. Furthermore, the equality implies that $\mathring{\nabla}^2 u = 0$, and the rigidity result follows as before.
 \end{demonstracao}

We finish this section with an additional manipulation of the boundary terms in Lemma \ref{LemmaReilly}, which leads to a Heintze-Karcher inequality in this context.

\begin{teorema}
Let \( M = I \times_{\rho} N \) be a warped product manifold, where \( N \) is an \((n-1)\)-dimensional Riemannian manifold and $I= [0,\infty)$ is an interval. Let \(\Sigma\) be a convex cone in \( M \) such that \(\Sigma \setminus \{O\}\) is smooth, and let \(\Omega \subset \Sigma\) be a sector-like domain such that exists a solution $u \in C^{1}(\Omega \cup \Gamma \cup \Gamma_1\setminus\{O\} ) \cap W^{1,\infty}(\Omega) \cap W^{2,2}(\Omega)$ to the problem \eqref{eq_soap} and the mean curvature $H$ of $\Gamma$ is positive. If $Ric \geq (n-1)kg,$ for some $k \in \Rset$ then  
\begin{equation*}
\frac{n-1}{n}\int_{\Gamma}\frac{1}{H}\;da\geq Vol(\Omega)+nk\int_{\Omega}u\;dv,    
\end{equation*}
and the equality holds if and only if $\Omega = \Sigma \cap B_{r}(x_0)$ where $B_{r}(x_0)$ is a metric ball ad $u$ is radial function.
\end{teorema}
\begin{demonstracao}
Note that
\begin{eqnarray}
-\frac{1}{n}\int_{\Gamma}u_{\nu}[(n-1)+nHu_{\nu}]\;da&=&\left(\frac{n-1}{n}\right)^{2}\int_{\Gamma}\frac{1}{H}\;da+\frac{n-1}{n}\int_{\Gamma}u_{\nu}\;da \nonumber\\
&-&\int_{\Gamma}\frac{1}{n^{2}H}[(n-1)+nHu_{\nu}]^{2}\;da \label{H-K1}
\end{eqnarray}
and, using \eqref{eq_soap},
\begin{equation}\label{H-K2}
\int_{\Gamma}u_{\nu}\;da=\int_{\partial\Omega}u_{\nu}\;da=-Vol(\Omega)-nk\int_{\Omega}u \;dv  
\end{equation}
Replacing \eqref{H-K1} and \eqref{H-K2} in \eqref{traceless} we obtain
\begin{eqnarray*}
 & &\frac{n}{n-1}\left\{\int_{\Omega}^{} |\mathring{\nabla}^2 u|^2\;dv+\int_{\Gamma}\frac{1}{n^{2}H}[(n-1)+nHu_{\nu}]^{2}\;da\right\}\\
 &\leq &\frac{n-1}{n}\int_{\Gamma}\frac{1}{H}\;da-Vol(\Omega)-nk\int_{\Omega}u\;dv. 
\end{eqnarray*}
Again, the two integrals on the left-hand side are nonnegative, and the inequality follows. Moreover, equality implies that $\mathring{\nabla}^2 u = 0$, and the rigidity result follows as before.
\end{demonstracao}

\begin{observacao}\label{rem_k=0}
We note that, in comparison with \cite{poggesi}, the results presented in this section may depend on \(u\). For instance, the constant \(H_0\) in Theorem \ref{TeoB} depends on \(c\), which, in turn, depends on \(u\). This dependency is expected in problems involving curvature (see analogous results in the Riemannian setting in \cite[Theorem 1, Theorem 4]{FAM_2023} and \cite[Theorem A]{AFS2024} for a weighted version). 

In any case, we highlight that the case \(k=0\), where there is no dependence on \(u\), is particularly interesting. Under the hypothesis \(\operatorname{Ric} \geq 0\), an inequality similar to that in \cite[Theorem 2.6]{poggesi} holds, and the constant \(H_0\) is given explicitly by 
\begin{equation*}
H_0 = \left(\frac{n-1}{n}\right)\frac{\operatorname{Area}(\Gamma)}{\operatorname{Vol}(\Omega)},
\end{equation*}
which is linked to the constant \(R\) in that result. Moreover, the above quantity is related to the well-known Cheeger constant; see \cite{cheeger}.

\end{observacao}

\section{The weighted Serrin's problem for convex cones of the Euclidean space}

Throughout this section, we consider $M = \mathbb{R}^{n}$ and recall the definition of an open cone $\Sigma$ in $\mathbb{R}^{n}$, $n \geq 2$, with vertex at the origin $O$: denoting by $\omega$ an open connected domain on the unit sphere $\mathbb{S}^{n-1}$, we define  
$$
\Sigma \doteq \left\{tx \,:\, x \in \omega, \,\, t \in (0, +\infty)\right\}.
$$  
Furthermore, a sector-like domain $\Omega$ is a domain where the boundary components are denoted by  
$$
\Gamma = \partial\Omega \cap \Sigma \quad \mbox{and} \quad \Gamma_{1} = \partial\Omega \setminus \overline{\Gamma},
$$  
with suitable properties (recall Definition \ref{def_sec}).

Given a positive smooth function $f$, the Euclidean space can modify its natural measure according to the following rule: the new volume element and surface area are given by $dv_{f} = f \, dv$ and $da_{f} = f \, da$, where $dv$ and $da$ represent the Euclidean volume and surface area elements for the canonical metric in $\mathbb{R}^{n}$ and $(\mathbb{R}^{n},\< \; , \;\>,  dv_{f})$ is a {\it weighted manifold}. In this setting,  the {\it drift Laplacian} (or {\it weighted Laplacian}) is the operator defined by
\[
\Delta_{f} \doteq \Delta + \langle \nabla \log f, \nabla \cdot \rangle.
\]
In this section, we study the problem
\begin{eqnarray} \label{SPRuan}
    \left\lbrace \begin{array}{ccccc}
	\Delta_f u  & = & -1 & \textrm{ in } & \Omega \\
	u & = & 0 & \textrm{ on } & \Gamma \\
        u_{\nu} & = & -c & \textrm{ on } & \Gamma \\
	    u_{\nu} & = & 0 & \textrm{ on } & \Gamma_1 \setminus \{O\} \\
        \nabla^2 \log f (\nabla u, \nabla u) + \displaystyle\frac{\< \nabla \log f, \nabla u \>^2 }{\alpha} &\leq & 0 & \textrm{ in } & \Omega
    \end{array} \right. 
\end{eqnarray}
where $\nu$ denote the exterior unit normal to $\partial\Omega$ and $f$ is a homogeneous function of degree $\alpha > 0,$ i.e. $\< \nabla f, x \> = \alpha f$, where $x$ is the position vector field. As a notable example we can take $f(x)=|x|^{\alpha}$, where $\alpha=2$ correspond to the Gaussian model. More than merely studying a new extension of the cone problem for a weighted operator, this section extends the analysis conducted in \cite{Ruan2018}. This, for example, explains the hypothesis imposed on the function $f$ and the final condition in \eqref{SPRuan}.
 
For the next lemma, we remember that a vector field $X\in\mathfrak{X}(M)$ is called \textit{closed conformal} if
\begin{equation*}
\nabla_Y X=\varphi Y,   \quad \text{ for all $Y\in\mathfrak{X}(M),$} 
\end{equation*}
for some smooth function $\varphi$ called the \emph{conformal factor}.

The next lemma, described in the recent work \cite{AFS2024}, plays an important role in the main result of this section. For the sake of completeness, we provide its proof here.

\begin{lema}\label{f_lapl} 
    Let $u$ be a smooth function on a Rimannian manifold $(M,g).$ Then, 
    \begin{equation*}
        \Delta_{f}\< X, \nabla u \> = \< \nabla \varphi , \nabla u \> (2-n) + 2 \varphi \Delta_f u + \< \nabla \Delta_f u, X \> - \nabla u (\< \nabla \log f, X \>)
    \end{equation*}
   where $X$ is a closed conformal vector field with conformal factor $\varphi.$
\end{lema}
\begin{demonstracao}
 Since $X$ is closed conformal vector field we conclude that  
 $$\Delta \< X, \nabla u \> = \< \nabla \varphi, \nabla u \> (2-n) + 2 \varphi \Delta u + \< \nabla \Delta u, X \>,$$
 see \cite[Lemma 2]{FAM_2023}. Thus, using the above equation and taking into account the definition of weighted Laplacian we have 
\begin{eqnarray*}
        \Delta_f \< X, \nabla u \> &=& \Delta \< X, \nabla u \> + \< \nabla \log f, \nabla \< X, \nabla u \> \> \\
        &=& \< \nabla \varphi, \nabla u \> (2-n) + 2 \varphi \Delta u + \< \nabla \Delta u, X \> + \< \nabla \log f, \varphi \nabla u + \nabla_X\nabla u \>.\\
        &=& \< \nabla \varphi, \nabla u \> (2-n) + 2 \varphi (\Delta_f u - \< \nabla \log f, \nabla u \>) + \< \nabla \Delta_f u, X \>\\ 
         &-&\< \nabla_X \nabla \log f, \nabla u \> - \< \nabla \log f ,\nabla_X \nabla u \> + \< \nabla \log f, \varphi \nabla u + \nabla_X\nabla u \>.
\end{eqnarray*}
On another hand, since the vector field $X$ is closed and conformal we get the following 
\begin{eqnarray*}
        \varphi \< \nabla \log f, \nabla u \> &=& \< \nabla \log f, \nabla_{\nabla u} X \> \;\;\;\;\;\;\;\;\;\;\;\;\;\;\;\;\;\;\;\;\;\;\;\;\;\;\;\;\;\;\;\; \\
        &=& \nabla u \< \nabla \log f, X \> - \< \nabla_{\nabla u} \nabla \log f, X \> \;\;\; \\
        &=& \nabla u \< \nabla \log f, X \> - \< \nabla_{X} \nabla \log f, \nabla u \> \;\;\; 
    \end{eqnarray*}
From above equations we obtain the desired result.
\end{demonstracao}

In our particular case, we obtain the following consequence.

\begin{corolario}\label{coro-eucl}
If $u$ is a solution of \eqref{SPRuan}, where $\Omega\subset\mathbb{R}^{n}$ and $f$ is homogeneous of degree $\alpha$, then $\Delta_{f}\< x, \nabla u \> = -2$, where $x$ is the position vector field.      
\end{corolario}

\begin{demonstracao}
The proof is completed by applying Lemma \ref{f_lapl}, observing that the position vector field $x$ is a closed conformal field with $\varphi = 1$. Moreover, since $f$ is $\alpha$-homogeneous, $\langle \nabla \log f, x \rangle = \alpha$.
\end{demonstracao}

A combination of the Bochner identity, the Cauchy-Schwarz inequality, and the Bergström inequality yields the following Bochner-type inequality for the drift laplacian, as stated in \cite[Appendix A]{wey} (see \cite[Section 1.5]{li} for an elementary proof of this result).
\begin{lema}\label{BochnerPonderadoGeral}
    Let $(M^n,g)$ be a Riemanninan manifold, $n\geq 3$, $u \in C^{\infty}(M)$, $\alpha>0$ and $f$ a  smooth positive function. Then
    \begin{eqnarray}\label{Boch_type}
        \frac{1}{2} \Delta_f |\nabla u|^2 \geq \frac{(\Delta_f u)^2}{n+\alpha} + \< \nabla u, \nabla \Delta_f u \> + Ric_{f}^{\alpha} (\nabla u, \nabla u).
    \end{eqnarray}
    Here $Ric_{f}^{\alpha}$ is the $\alpha$-Bakry-Émery-Ricci tensor given by
\begin{equation*}
    Ric_{f}^{\alpha} = Ric - \nabla^2 \log f - \frac{1}{\alpha} d\log f \otimes d \log f. 
\end{equation*}
    Furthermore, the equality holds in \eqref{Boch_type} if and only if
    \begin{eqnarray*}
        && \nabla^2 u = \frac{\Delta u}{n} g, \\ 
        && \frac{\Delta u}{n} = \frac{\< \nabla \log f, \nabla u \>}{\alpha} = \frac{\Delta_f u}{n + \alpha}.
    \end{eqnarray*}
\end{lema}
Applying the last result to our specific problem, we obtain the following rigidity characterization.
\begin{corolario}\label{coro-Bochner}
If $u$ is a solution of \eqref{SPRuan}, where $\Omega\subset\mathbb{R}^{n}$, then 
$$\frac{1}{2} \Delta_f |\nabla u|^2 \geq\frac{1}{n+\alpha}.$$
Furthermore, the equality holds if and only if $\Omega = \Sigma \cap B_{r}(x_0)$ where $B_{r}(x_0)$ is a ball.
\end{corolario}
\begin{demonstracao}
Since $Ric=0$ in $\mathbb{R}^{n}$, we use \eqref{SPRuan} to conclude
$$Ric_{f}^{\alpha} (\nabla u, \nabla u)=-\nabla^2 \log f (\nabla u, \nabla u)- \displaystyle\frac{\< \nabla \log f, \nabla u \>^2 }{\alpha}\geq 0.$$
To conclude the result we just replace the solution $u$ in Lemma \ref{BochnerPonderadoGeral}. The conclusion of the equality follows from Lemma \ref{rigidez_Tipo_Obata}.
\end{demonstracao}

We conclude this section with the proof of its main result, Theorem \ref{teoC}.
\vspace{0.2cm}

\begin{demonstracao}[Proof of Theorem \ref{teoC}]
    Since $u$ a solution of \eqref{SPRuan}, 
    \begin{eqnarray*}
        \int_{\Omega}^{} u\, dv_f &=& - \int_{\Omega}^{} u \Delta_f u\, dv_f  \\
        &=& - \int_{\Gamma}^{} u u_{\nu}\, da_f  - \int_{\Gamma_1}^{} u u_{\nu}\, da_f  + \int_{\Omega}^{} |\nabla u|^2 \, dv_f \\
        &=& - \int_{\Omega}^{} |\nabla u|^2 \Delta_f u \, dv_f \\
        &=& - \int_{\Gamma\cup\Gamma_{1}}^{} |\nabla u|^2 u_{\nu} da_f  + \int_{\Omega}^{} \< \nabla |\nabla u|^2, \nabla u \>\, dv_f  \\
        &=& - \int_{\Gamma}^{} |\nabla u|^2 u_{\nu}\, da_f + \int_{\Omega}^{} \< \nabla |\nabla u|^2, \nabla u \>\, dv_f.\\
        &=& - c^2\int_{\Gamma}^{} u_{\nu}\, da_f + \int_{\Omega}^{} \< \nabla |\nabla u|^2, \nabla u \>\, dv_f.\\
         &=& - c^2\int_{\partial\Omega}^{} u_{\nu}\, da_f + \int_{\Omega}^{} \< \nabla |\nabla u|^2, \nabla u \>\, dv_f.
    \end{eqnarray*}
    Then,
    \begin{equation}\label{ineq}
      \int_{\Omega}u\, dv_{f}=c^{2}Vol_{f}(\Omega)+ \int_{\Omega}^{} \< \nabla |\nabla u|^2, \nabla u \>\, dv_f. 
    \end{equation}
On other hand, using the divergence theorem,
\begin{eqnarray*}
\int_{\Omega} \< \nabla |\nabla u|^2, \nabla u \>\, dv_f&=&-\int_{\Omega}u\Delta_{f}|\nabla u|^{2}\;dv_{f}+\int_{\partial\Omega} u\langle\nabla |\nabla u|^2,\nu\rangle\; da_{f}\\
&\leq& -\frac{2}{n+\alpha}
\int_{\Omega} u\; dv_{f}+\int_{\Gamma_{1}}u\langle\nabla |\nabla u|^2,\nu\rangle \;da_{f}\\
&\leq & -\frac{2}{n+\alpha}
\int_{\Omega} u\; dv_{f},
\end{eqnarray*}
where, in the first inequality, we use Corollary \ref{coro-Bochner}, and in the second inequality, we use the convexity of $\Sigma$ (see a similar calculation in the proof of Theorem \ref{TeoA}). Then,

    Then, \eqref{ineq} becomes
    \begin{eqnarray}
    \label{novaeq012}
        \Big( 1 + \frac{2}{n+\alpha} \Big) \int_{\Omega}^{} u \, dv_f\leq c^2 Vol_{f}(\Omega),
    \end{eqnarray}
    being the equality characterized by Corollary \ref{coro-Bochner}.
    
    Using the $\alpha$-homogeneity of $f$, we have that
    \begin{eqnarray*}
        \Div_f(ux) &=& u \Div_f(x) + \< \nabla u, x \> \\
        &=& u (\Div (x) + \< x ,\nabla \log f \>) + \< \nabla u, x \> \\
        &=& un+ u \< x ,\nabla \log f \> + \< \nabla u, x \> \\
        &=& (n+\alpha) u + \< \nabla u, x \>,
    \end{eqnarray*}
    therefore
    \begin{eqnarray*}
        \int_{\Omega}^{} (n+\alpha) u \, dv_f &=& \int_{\Omega}^{} \Div_f(ux) \, dv_f-\int_{\Omega}^{} \< \nabla u, x \>\; dv_f .
    \end{eqnarray*}
    Since
    \begin{eqnarray*}
        \int_{\Omega}^{} \Div_f(ux)\; dv_f  &=& \int_{\partial \Omega}^{} u \< x ,\nu \> \; da_f \\ 
        &=& \int_{\Gamma}^{} u \< x, \nu \>\; da_f  + \int_{\Gamma_1}^{} u  \< x, \nu \> \; da_f \\
        &=& 0,
    \end{eqnarray*}
    we get
     \begin{eqnarray*}
         (n+\alpha)\int_{\Omega}^{}  u \, dv_f &=&- \int_{\Omega}^{} \< \nabla u, x \>\, dv_f. 
     \end{eqnarray*}
    Furthermore we observe that $\nabla^2u(\nabla u,x)=0$ on $\Gamma_1.$ Indeed, note that
     \begin{eqnarray}\label{calculation}
         0&=&x\langle\nabla u,\nu\rangle\\
    &=&\langle\nabla_x\nabla u,\nu\rangle
    +\langle \nabla u,\nabla_x\nu\rangle,\nonumber
     \end{eqnarray}
on $\Gamma_1.$
Now, given a vector field $Y\in\mathfrak{X}(\Gamma_1)$, note that
\begin{eqnarray*}
\langle\nabla_x\nu,Y\rangle &=& \langle\nu,\nabla_xY\rangle\\
&=&\langle\nu,\nabla_Yx+[x,Y]\rangle\\
&=&\langle\nu,Y+[x,Y]\rangle\\
&=& 0.
\end{eqnarray*}

Thus, from above equation, we deduce that $\nabla_x\nu=0$ and from \eqref{calculation} we conclude that $\nabla^2u(\nabla u,x)=0.$

\medskip
      
Since $\nabla^2u(\nabla u,x)=0,$ on $\Gamma_1$, from integration by parts and using Corollary \ref{coro-eucl}, we have 
    \begin{eqnarray*}
        \int_{\Omega}^{} u\, dv_f &=& \frac{1}{n+\alpha}\Big(  - \int_{\Omega}^{} \< x, \nabla u \>\, dv_f \Big)\\
        &=&\frac{1}{n+\alpha}\Big( \int_{\Omega}^{} \< \nabla u, x \>\Delta_{f}u\, dv_f \Big)\\
        &=& \frac{1}{n + \alpha} \Big ( \int_{\partial\Omega}^{} \< x , \nabla u \> u_{\nu}\; da_f + \int_{\Omega}^{} u \Delta_f(\< x, \nabla u \>)\, dv_f -\int_{\partial\Omega}u\langle\nu,\nabla\langle x,\nabla u\rangle\rangle\; da_f \Big) \\
        &=& \frac{1}{n+\alpha} \Big ( \int_{\Gamma}^{} \< x , \nabla u \> u_{\nu} \;da_f -\int_{\Gamma_{1}} u\langle\nu,\nabla\langle x,\nabla u\rangle\rangle\; da_f  -2 \int_{\Omega}^{} u\; dv_f \Big) \\
        &=& \frac{1}{n+\alpha} \Big( c^2 \int_{\partial\Omega}^{} \< x ,\nu \>\, da_f  -2 \int_{\Omega}^{} u\; dv_f \Big) \\
        &=& \frac{1}{n + \alpha} \Big( c^2 \int_{\Omega}^{} \Div_f(x)\, dv_f - 2 \int_{\Omega}^{} u\, dv_f \Big)\\
        &=&  c^{2}Vol_{f}(\Omega)- \frac{2}{n+\alpha} \int_{\Omega}^{} u\, dv_f
    \end{eqnarray*}
    therefore occurs the equality in \eqref{novaeq012}. This implies, by Corollary \ref{coro-Bochner} that $\Omega = \Sigma \cap B_{r}(x_0)$ where $B_{r}(x_0)$ is a ball. 
    
    Now, we observe that from Lemma \ref{BochnerPonderadoGeral}, we get
 $ \nabla^2 u = \frac{\Delta u}{n} g= 
         -\frac{1}{n + \alpha}g.$ Thus, after some routine calculation one arrives at $u=r^2-\frac{|x|^2}{2(n+\alpha)}.$

\end{demonstracao}

\begin{observacao}
The authors believe that the results of Section 4 could be extended to domains in Riemannian manifolds endowed with a closed conformal vector field, where a condition on $\alpha$-Bakry-Émery Ricci curvature bounded below would allow an extension of Corollary \ref{coro-Bochner} and its rigidity conclusion (see, for example, the condition explicitly stated in \cite[Theorem B]{AFS2024}). While this represents an interesting (and natural) improvement of those results, we prefer to emphasize the Euclidean case and its visible consequences.  
\end{observacao}

\section*{Declarations} 
\subsection*{Ethics approval and consent to participate}
Not applicable
\subsection*{Consent for publication}
Not applicable
\subsection*{Data Availability}
Data sharing was not applicable to this article as no datasets were generated or analyzed during the current study.
\subsection*{Competing interests}
The authors declare no competing interests.
\subsection*{Authors' contributions}
All authors wrote the main manuscript text and reviewed the manuscript.

\section*{Funding}
The second and third authors have been partially supported by Conselho Nacional de Desenvolvimento Científico e Tecnológico (CNPq) of the Ministry of Science, Technology and Innovation of Brazil, Grants 316080/2021-7 and 200261/2022-3 (A. Freitas) and 306524/2022-8 (M. Santos), and supported by Paraíba State Research Foundation(FAPESQ), Brazil, Grant 3025/2021 (A. Freitas and M. Santos). The second author also was funded by Paraíba State Research
Foundation - Programa Primeiros Projetos (Grant 2021/3175) and CIMPA/ICTP Research in Pairs 2024.

\section*{Acknowledgements}
Allan Freitas extends his gratitude to the Mathematics Department of Università degli Studi di Torino and ICTP - International Centre for Theoretical Physics for their hospitality during the completion of part of this work. He is especially thankful to Luciano Mari for his warm hospitality and continuous encouragement.

Murilo Araújo expresses his gratitude to UFAPE and his colleagues in the Department of Mathematics for their support, which has enabled him to pursue his doctoral studies at UFPB with full dedication. This work constitutes a part of his Ph.D. thesis.

\end{document}